\documentclass[11pt,reqno]{amsart}
\usepackage{amscd,amsmath,amsopn,amssymb,amsthm,multicol}
\usepackage{tikz,subdepth,anysize,verbatim,ifthen,xargs,colortbl,float}
\usepackage{longtable,mathtools,hyperref}
\usepackage[english]{babel}
\everymath=\expandafter{\the\everymath\displaystyle}
\theoremstyle{plain}
\newtheorem{theorem}{Theorem}

\newtheorem{lemma}{Lemma}
\newtheorem{cor}{Corollary}
\newtheorem{rk}{Remark}
\theoremstyle{definition}
\newtheorem{definition}{Definition}
\newtheorem{example}{Example}

\newcommand\com[1]{}
\newcommand\C{{\mathbb C}}
\newcommand\f{{\mathfrak f}}
\newcommand\g{{\mathfrak g}}
\newcommand\gl{{\mathfrak{gl}}}
\newcommand\h{\mathfrak{h}}
\newcommand\p{{\mathfrak p}}
\newcommand\m{{\mathfrak m}}

\newcommand\op[1]{\mathop{\rm #1}\nolimits}
\newcommand\R{{\mathbb R}}
\renewcommand\sl{{\mathfrak{sl}}}
\newcommand\Z{{\mathbb Z}}

\usetikzlibrary{arrows,decorations.markings,positioning}
\tikzset{inner sep=0pt, node distance=5mm,
  root/.style={circle,draw,minimum size=5pt,thick},
  broot/.style={circle,draw,minimum size=5pt,thick,fill},
  xroot/.style={circle,draw,minimum size=5pt,thick,fill=gray!70!white},
  crossroot/.style={circle,draw,minimum size=5pt,thick,label=below:$\times$},
  doublearrow/.style={postaction={decorate},   decoration={markings,mark=at position .6 with {\arrow[line width=1.2pt]{>}}},double distance=1.6pt,thick},
  doublenoarrow/.style={double distance=1.6pt,thick},
  rdoublearrow/.style={postaction={decorate},   decoration={markings,mark=at position .4 with {\arrowreversed[line width=1.2pt]{>}}},double distance=1.6pt,thick},
  rtriplearrow/.style={postaction={decorate},   decoration={markings,mark=at position .4 with {\arrowreversed[line width=1.2pt]{>}}},double distance=2.5pt,thick},
  curvedline/.style={bend=right}
}

\begin{document}

\title{Prolongation rigidity of sub-free Lie algebras}

\author{Boris Kruglikov}
\address{Department of Mathematics and Statistics, UiT the Arctic University of Norway, Troms\o\ 9037, Norway.
\ E-mail: {\tt boris.kruglikov@uit.no}. }

 \begin{abstract}
We prove that if the 0-th Tanaka prolongation $\g_0=\mathfrak{der}_0(\m)$ of a fundamental graded 
nilpotent Lie algebra $\m=\g_{-s}\oplus\dots\oplus\g_{-1}$ is irreducible on $\g_{-1}$, 
then $\m$ is prolongation rigid: $\op{pr}_+(\m)=0$. 
The only exceptions are given by negative gradations of maximal parabolic subalgebras of a simple Lie algebra.
 \end{abstract}

\maketitle

%%%%%
\section{Formulation of the Problem and the Results}\label{S1}

A finite-dimensional graded nilpotent Lie algebra (GNLA) $\m=\g_{-s}\oplus\dots\oplus\g_{-1}$ is called fundamental, 
if $\g_{-1}$ bracket-generates $\m$ and contains no central elements $\mathfrak{z}(\m)\cap\g_{-1}=0$.
The depth of $\m$ is $s$. If $\m$ is fundamental then $s>1$ and $\g_{-1}$ is proper (otherwise $\mathfrak{z}(\m)=\g_{-1}=\m$); 
the subspace $\g_{-1}$ corresponds to a nonholonomic distribution $\Delta$ on $M_o=\exp\m$.

The Tanaka prolongation \cite{T} of $\m$ is the maximal graded Lie algebra 
$\g=\op{pr}(\m)=\bigoplus_{k=-s}^\infty\g_k$ such that $\m=\g_-:=\oplus_{k<0}\g_k$
and $[\g_{-1},x]=0$ for $x\in\g_{\ge0}:=\oplus_{k\geq0}\g_k$ implies $x=0$
(we will also write $\g_k=\op{pr}_k(\m)$ for $k\geq0$). 
This can be constructed inductively in grading $k\geq0$.
In particular, $\g_0=\mathfrak{der}_0(\m)$ consists of grading preserving derivations (it always contains the
grading element $Z$ given by $[Z,x]=kx$ for $x\in\g_k$) and $\g_1$ is identified with the space of elements in 
$\op{Hom}(\g_{-1},\g_0)$ that (uniquely) extend to a homomorphism $\m\to\m\oplus\g_0$ of degree $1$
(that is, it maps $\g_i\to\g_{i+1}$ for $i<0$) 
satisfying the Jacobi identity, whenever defined.
 % then $\g_2$ is the space of elements in $\op{Hom}(\g_{-1},\g_1)$ that (uniquely) extend
 % to a homomorphism $\m\to\g_{\leq1}$ of degree $2$, etc. 
If some $\g_l=0$ then $\g_k=0$ for all $k\geq l$;
such $\m$ are called finite type or rigid.

A free graded nilpotent Lie algebra $\f(n)$ with $n$ generators is a maximal Lie algebra generated by 
$\g_{-1}$, $\dim\g_{-1}=n$, modulo the constraints from the Jacobi identity.
Any GNLA is a quotient of the free GNLA by a graded ideal $\h\subset\g$, $\h_{-1}=0$.
A truncated free GNLA of depth $s$ is a free GNLA $\f_s(n)$ that is obtained from $\f(n)$ as the quotient by the ideal
of gradation less than $-s$: $\f_s(n)=\oplus_{i=1}^s\g_{-i}$. We will also call 
$\f_s(n)$ a free graded nilpotent Lie algebra of depth $s$.

It is easy to see that $\g_0=\op{pr}_0(\f_s(n))=\gl(n)$ (our results apply both over $\R$ and $\C$, 
so we will not specify the ground field).
In this case it was proven in \cite{W} that $\op{pr}_+(\f_s(n)):=\oplus_{k>0}\g_k=0$, 
so that $\op{pr}(\m)=\f_s(n)\rtimes\gl(n)$, provided the depth $s>2$ for $n>2$ and $s>3$ for $n=2$.
We will generalize this result.

 \begin{definition}
A fundamental GNLA $\m$ is called sub-free if $\g_0=\op{pr}_0(\m)=\gl(n)$, where $n=\dim\g_{-1}$.
(In particular, a free GNLA is also sub-free.)
 \end{definition}

As any GNLA with $n$ generators, a sub-free GNLA is a quotient algebra of a free GNLA of depth $s$, however we additionally
require that the 0-th prolongation is preserved in this reduction. 

We recall that for a simple Lie algebra $\g$, a parabolic subalgebra $\p$ is the nonnegagive part $\g_0\oplus\g_+$
of a $\Z$-grading of $\g$ (such grading is bijective with a choice of parabolic, up to conjugation). 
For complex Lie algebras, the parabolic is marked by crossed nodes of the Dynkin diagram, 
corresponding to simple root vectors not in $\g_0$.
Over $\R$ the same applies for the Satake diagram, with the constraint that black nodes cannot be crossed and
the arrow-related nodes must be crossed simultaneously. We denote by $\p_I$ the parabolic corresponding to crossed nodes
$I\subset\{1,\dots,r\}$ for $r=\op{rank}(\g)$. We use Bourbaki's numeration of simple roots \cite{B}.

The negative part $\m_I$ is isomorphic to the nilradical of $\p_I$ (equivalently: $\m_I$ is the nilradical of the opposite parabolic), 
so $\g=\m\oplus\p$. The Yamaguchi theorem \cite{Ym} ensures when $\op{pr}(\m)=\g$.
One of our main results is:

 \begin{theorem}\label{Th1}
Let $\m$ be a fundamental GNLA different from $\m_I$ in the case of parabolic $\p_1$ for $G_2$ or 
$\p_r$ for $B_r$. If $\m$ is sub-free, then it is prolongation rigid: $\op{pr}_+(\m)=0$.
 \end{theorem}

The proof is based on structure theory of Lie algebras; it is independent of and is shorter than 
the prolongation rigidity proof for free Lie algebras in \cite{W}. We will generalize this further to 
the notion of $\g_0$ sub-free Lie algebras and the corresponding rigidity result in Section \ref{S3}.
 
To discuss geometric applications, let us recall basics of vector distributions.
For a nonholonomic distribution $\Delta\subset M$ its symbol at a point $x\in M$ is the GNLA associated to the 
filtered weak derived flag 
 $$
\Delta=\Delta^{1}\subset\Delta^{2}=[\Delta,\Delta]\subset\Delta^{3}=[\Delta,\Delta^{2}]\subset\Delta^{4}=[\Delta,\Delta^{3}]
\subset\dots
 $$
where the brackets of distributions stand for the brackets of the corresponding modules of sections; if the resulting 
$C^\infty(M)$-module is projective (which we adapt) it is the module of sections of the resulting distribution.
Thus $\Gamma([D,D'])=[\Gamma(D),\Gamma(D')]$, etc. Then $\g_{-1}(x)=\Delta_x$ and
$\g_{-i}(x)=\Delta^i_x/\Delta^{i-1}_x$ for $i>1$. So we get a GNLA $\m(x)=\oplus_{k<0}\g_k(x)$ for every $x\in M$,
which is fundamental provided $\Delta$ is nonholonomic (bracket generating $TM$) and free of Cauchy characteristics 
(symmetries from $\Gamma(\Delta)$ corresponding to $\mathfrak{z}(\m)\cap\g_{-1}$); 
then we compute its Tanaka prolongation $\g(x)=\op{pr}(\m(x))$. 

According to \cite{K} the symmetry algebra $\op{Sym}(\Delta)$ of $(M,\Delta)$ is bounded by $\sup_{x\in M}\dim\g(x)$.

 \begin{cor}\label{Cor1}
Let $\Delta$ be a nonholonomic distribution of rank $n$ on a manifold $M$ of dimension $m$. 
If its symbol $\m(x)$ is sub-free at every point $x\in M$ and differs from exceptions of Theorem \ref{Th1}, 
then $\dim\op{Sym}(\Delta)\leq n^2+m$.
 \end{cor}

A Lie group $G_x=\exp\m(x)$ is called the Carnot group or nilpotent approximation of $(M,\Delta)$ at $x$.
Its rigidity, i.e.\ finite-dimensionality of the symmetry algebra (or Tanaka prolongation), is important in control theory
\cite{R,OW,GKMV}. The prolongation rigidity of Theorem \ref{Th1} is an even stronger property, sufficient for rigidity.

%%%%%
\section{Examples of sub-free Lie algebras}\label{S2}

Free Lie algebras are given by $\f(n)=\dots\oplus\g_{-2}\oplus\g_{-1}$. The string $(\dim\g_{-1},\dim\g_{-2},\dots)$
is called the growth vector of $\m$ (in geometric terms, the reduced growth of distribution $\Delta$). 
The formula for $d_k=\dim\g_{-k}$, with $d_1=n$ is well-known \cite{Re}:
 $$
d_k=\frac1k\sum_{t|k}\mu(t)n^{k/t}, 
 $$
where $\mu(t)$ is the M\"obius function. Equivalently, by the M\"obius inversion, $d_k$ can be found from the recursive relation 
$\sum_{t|k}td_t=n^k$. 
For $n=2$ the reduced growth vector of $\f(n)$ is
 $$
2, 1, 2, 3, 6, 9, 18, 30, 56, 99, 186, 335, 630, 1161, 2182, 4080, 7710, 14532, 27594, 52377, \dots 
 $$
for $n=3$ it is 
 $$
3, 3, 8, 18, 48, 116, 312, 810, 2184, 5880, 16104, 44220, 122640, 341484, 956576, \dots 
 $$
for $n=4$ it is 
 $$
4, 6, 20, 60, 204, 670, 2340, 8160, 29120, 104754, 381300, 1397740, \dots 
 $$
etc. (For Lie algebras $\f_s(n)$ one has to truncate the string.)

Since $\g_0=\gl(n)$ the vector spaces $\g_{-k}$ are $\gl(n)$ modules. Both $\g_{-1}$ and $\g_{-2}$ are irreducible.
Using notations $\pi_i$ for fundamental weights of $A_{n-1}=\sl(n)$ and $\Gamma_\omega$ for representations
of weight $\omega=\sum p_i\pi_i$, they are respectively the standard module $\Gamma_{\pi_1}$ and the second 
fundamental module $\Gamma_{\pi_2}=\Lambda^2\Gamma_{\pi_1}$ (except for $n=2$ when we get $\Gamma_0$). 

In order to describe the components of free and sub-free algebras we need the following branching rule for arising $\g_0$ modules.
(The Jacobi identity corresponds to $\Lambda^3\g_{-1}$ and more complicated plethysms.)

 \begin{lemma}
For arbitrary GNLA we have $\g_{-2}\subset\Lambda^2\g_{-1}$ and $\g_{-k-1}\subset\g_{-1}\otimes\g_{-k}$ for $k>1$.
 \end{lemma}
 
 \begin{proof}
The bracket generating property yields an epimorphism $\g_{-1}\otimes\g_{-k}\to\g_{-k-1}$ (for $k=1$ we may use the wedge product). 
Since all $\gl(n)$ modules are completely reducible, we can split and invert the arrow. 
 \end{proof} 

{\bf Case n\,$=$\,2.} In this case the irreps of $\g_0^{ss}=A_1$ are enumeratred by $\Gamma_t:=\Gamma_{t\pi_1}$ (of $\dim=t+1$) and the free GNLA has the following description. 
We have $\g_{-1}=\Gamma_1$, $\g_{-2}=\Gamma_0$, $\g_{-3}=\Gamma_1$. Then $\g_{-1}\otimes\g_{-3}=\Gamma_0+\Gamma_2$, 
and the first component is ruled out by the Jacobi identity, so that $\g_{-4}=\Gamma_2$. 
Next, $\g_{-1}\otimes\g_{-4}=\Gamma_1+\Gamma_3=\g_{-5}$.
Next, $\g_{-1}\otimes\g_{-5}=\Gamma_1\otimes\Gamma_1+\Gamma_1\otimes\Gamma_3=
\Gamma_0+2\Gamma_2+\Gamma_4$ and the Jacobi identity eliminates one copy of $\Gamma_2$, so that 
$\g_{-6}=\Gamma_0+\Gamma_2+\Gamma_4$. Then we have
$\g_{-1}\otimes\g_{-6}=2\Gamma_1+2\Gamma_3+\Gamma_5=\g_{-7}$, etc.

A reduction of one or more of those components, compatible with bracket generation, gives an example of a sub-free Lie algebra.
Consider, for instance, $\g_{-5}$ that contains two proper submodules $\Gamma_1$ and $\Gamma_3$ these can be taken as reductions
$\g_{-5}'$ and $\g_{-5}''$, respectively, and we obtain two sub-free GNLA of depth 5: 
 $$
\m'_5= \g_{-5}'\oplus\g_{-4}\oplus\g_{-3}\oplus\g_{-2}\oplus\g_{-1}\quad\text{and}\quad
\m''_5= \g_{-5}''\oplus\g_{-4}\oplus\g_{-3}\oplus\g_{-2}\oplus\g_{-1}.
 $$
These have growth $(2,1,2,3,2)$ and $(2,1,2,3,4)$, respectively, and have the following common relations
 $$
\m_4: \qquad  
[e^1_1,e^1_2]=e^2,\ [e^1_1,e^2]=e^3_1,\ [e^1_2,e^2]=e^3_2,\ [e^1_1,e^3_1]=e^4_{11},\ 
[e^1_1,e^3_2]=[e^1_2,e^3_1]=e^4_{12},\ [e^1_2,e^3_2]=e^4_{22}  
 $$
and then specific structure relations ($e^k_\sigma$ is a basis element of grading $-k$ and numeration $\sigma$):
 \begin{align*}
\m'_5: & \quad 
[e^1_1,e^4_{12}]=\tfrac23e^5_1,\ [e^1_1,e^4_{22}]=\tfrac43e^5_2,\ 
[e^1_2,e^4_{11}]=-\tfrac43e^5_1,\ [e^1_2,e^4_{12}]=-\tfrac23e^5_2,\
[e^2,e^3_1]=2e^5_1,\ [e^2,e^3_2]=2e^5_2; \\
\m''_5: & \quad
[e^1_1,e^4_{11}]=e^5_{111},\ [e^1_1,e^4_{12}]=[e^1_2,e^4_{11}]=e^5_{112},\ 
[e^1_1 e^4_{22}]=[e^1_2,e^4_{12}]=e^5_{122},\ [e^1_2,e^4_{22}]=e^5_{222}.
 \end{align*}
By Theorem \ref{Th1} their prolongations are $\op{pr}(\m'_5)=\m'_5\rtimes\gl(2)$ and $\op{pr}(\m''_5)=\m''_5\rtimes\gl(2)$.
Choosing one of those or $\f_5(2)$ we can include a part of $\g_{-6}$ that is bracket generated, for instance 
 $$
\m'_6= \g_{-6}'\oplus\m'_5\quad\text{and}\quad\m''_6= \g_{-6}''\oplus\m'_5,
 $$
where $\g_{-6}'\subset\Gamma_0+\Gamma_2$ and $\g_{-6}''\subset\Gamma_2+\Gamma_4$, etc.

\smallskip

{\bf Case n\,$>$\,2.} 
In the case $n>2$ we have for representations of $A_{n-1}$: 
$\g_{-1}\otimes\g_{-2}=\Gamma_{\pi_1}\otimes\Gamma_{\pi_2}=\Gamma_{\pi_3}+\Gamma_{\pi_1+\pi_2}$.
The first summand (that should be $\Gamma_0$ for $n=3$) is ruled out by the Jacobi identity, so we get
$\g_{-3}=\Gamma_{\pi_1+\pi_2}$. 

Next, $\g_{-1}\otimes\g_{-3}=\Gamma_{\pi_1}\otimes\Gamma_{\pi_1+\pi_2}=
\Gamma_{2\pi_2}+\Gamma_{\pi_1+\pi_3}+\Gamma_{2\pi_1+\pi_2}$ and the Jacobi identity eliminates the first term, so that
$\g_{-4}=\Gamma_{\pi_1+\pi_3}+\Gamma_{2\pi_1+\pi_2}$ for $n>3$ and $\g_{-4}=\Gamma_{\pi_1}+\Gamma_{2\pi_1+\pi_2}$ for $n=3$. 

Then $\g_{-1}\otimes\g_{-4}=\Gamma_{\pi_2}+2\Gamma_{2\pi_1}+\Gamma_{\pi_1+2\pi_2}+\Gamma_{3\pi_1+\pi_2}$ for $n=3$,
$\g_{-1}\otimes\g_{-4}=\Gamma_{\pi_1}+\Gamma_{\pi_2+\pi_3}+2\Gamma_{2\pi_1+\pi_3}+\Gamma_{\pi_1+2\pi_2}+\Gamma_{3\pi_1+\pi_2}$ for $n=4$,
and
$\g_{-1}\otimes\g_{-4}=\Gamma_{\pi_1+\pi_4}+\Gamma_{\pi_2+\pi_3}+2\Gamma_{2\pi_1+\pi_3}+\Gamma_{\pi_1+2\pi_2}+\Gamma_{3\pi_1+\pi_2}$ for $n>4$. The Jacobi identity eliminates one multiple term (in each case), so that
$\g_{-5}=\Gamma_{\pi_2}+\Gamma_{2\pi_1}+\Gamma_{\pi_1+2\pi_2}+\Gamma_{3\pi_1+\pi_2}$ for $n=3$,
$\g_{-5}=\Gamma_{\pi_1}+\Gamma_{\pi_2+\pi_3}+\Gamma_{2\pi_1+\pi_3}+\Gamma_{\pi_1+2\pi_2}+\Gamma_{3\pi_1+\pi_2}$ for $n=4$,
and
$\g_{-5}=\Gamma_{\pi_1+\pi_4}+\Gamma_{\pi_2+\pi_3}+\Gamma_{2\pi_1+\pi_3}+\Gamma_{\pi_1+2\pi_2}+\Gamma_{3\pi_1+\pi_2}$ for $n>5$, etc.

For instance, for $n=3$ the simplest sub-free GNLA $\m'_4$ that is not (truncated) free has depth 4 and growth vector $(3,3,8,3)$.
The graded components of $\m$ have bases (indices have range $1\leq i,j,k\leq3$)
$e^1_i$ for $\g_{-1}$,  $e^2_{jk}$ for $\g_{-2}$ (skew-symmetric in $[jk]$),  
$e^3_{ijk}$ for $\g_{-3}$ (also skew in $[jk]$ plus a constraint $e^3_{123}+e^3_{231}+e^3_{312}=0$)
and $e^4_k$ for $\g_{-4}$. The structure relations of $\m'_4$ are ($\epsilon_{ijk}$ is the volume density in 3D; summation by repeated $r$)
 % note the last relation is skew in pairs $ij$ and $kl$
 $$
[e^1_i,e^1_j]=e^2_{ij},\quad [e^1_i,e^2_{jk}]=e^3_{ijk},\quad [e^1_l,e^3_{ijk}]=\epsilon_{ljk}e^4_i-\delta_{il}\epsilon_{rjk}e^4_r,\quad 
[e^2_{ij},e^2_{kl}]=\epsilon_{ikl}e^4_j-\epsilon_{jkl}e^4_i.
 $$
Similarly, another sub-free GNLA $\m''_4$ of depth 4 is obtained from the above $\f_4(3)$ by quotienting out $\Gamma_{\pi_1}$;
it has growth $(3,3,8,15)$. By Theorem \ref{Th1} the prolongations are $\op{pr}(\m'_4)=\m'_4\rtimes\gl(3)$ and 
$\op{pr}(\m''_4)=\m''_4\rtimes\gl(3)$.

Consequently, for each $n$, we get a tree of sub-free Lie algebras, obtained by reduction of some irreducible components 
with branching according to the depth, where we may truncate GNLA at any level.

%%%%%
\section{Prolongation rigidity and generalization}\label{S3}

 \begin{proof}[Proof of Theorem \ref{Th1}]
We first claim that $\m$ is rigid, i.e.\ its prolongation is finite-dimensional. Since the prolongation is a linear algebra procedure, 
this property does not depend on whether we work over $\R$ or $\C$, so we may complexify $\m$ to demonstrate the claim. 
Whether a complex GNLA $\m$ is rigid is decided by the corank one criterion of \cite[Theorem 3.1]{DR}. 
It states that $\m$ is of infinite type ($\dim\op{pr}(\m)=\infty$) if and only if 
 \[
\exists x\in\g_{-1}\setminus0:\quad \op{rank}(\op{ad}_x|\m)=1.
 \]
In other words, there exist % $x\in\g_{-1}\setminus0$ and 
a hyperplane $\Pi\subset\m$ (graded subspace), $x\in\Pi$, such that $[x,y]=0$ for all $y\in\Pi$. 

Since  by assumptions $\g_{-2}=\Lambda^2\g_{-1}$ (irrep)
we conclude that $[x,\g_{-1}\cap\Pi]\neq0$ for all nonzero $x\in\g_{-1}\cap\Pi$ if $n>2$. 
If $n=2$ then $\m$ is of depth $s=2$ and $\m=\g_-=\mathfrak{heis}(3)$ for the Borel gradation of $\g=\sl(3)$
and also for the $\p_2$ gradation of $B_2$. Indeed, in the opposite case, if $n=2$ and $s>2$ then
$\g_{-3}\simeq\g_{-1}$ so $\g_{-2}\subset\Pi$ and $[x,\g_{-2}]\neq0$. We conclude that $\m$ is of finite type.
 
Thus $\g=\op{pr}(\m)$ is a finite-dimensional Lie algebra. Let $\g^{ss}$ be the semi-simple part in the Levi decomposition 
of $\g$. It clearly contains $\sl(n)$. By Proposition 2.5 of \cite{Yt} $\g^{ss}$ contains the entire positive part 
$\op{pr}_+(\m)$ of the prolongation. Assume $\m$ is not prolongation rigid, so $\g^{ss}\neq\sl(n)$. 
In this case $\g^{ss}$ must also contain the dual $\g_0$-modules for irreps from $\g_1$. 
Since $\g_{-1}$ is irreducible, we obtain $\g_{-1}\subset\g^{ss}$. However $\g_{-1}$ bracket-generates $\m$, 
so we conclude that $\g=\g_{-s}\oplus\dots\oplus\g_s$ is semi-simple, where $s$ is the depth of $\m$.

Moreover, due to $\g_0=\gl(n)$ and the prolongation property, $\g$ is actually simple. The only gradings of simple Lie algebras, 
where $\g_0$ is of the form $\gl(k)$ for some $k$ and the depth is $s>1$, are given in the following table. 
We indicate the growth vector $(\dim\g_{-1},\dots,\dim\g_{-s})$ of $\m=\g_-$ 
together with $\g_0^{ss}=A_{n-1}$ representation types of graded components $\g_{-k}$.

 \begin{table}[h!]
\centering
 \begin{tabular}{c|c|c|c} 
 \hline
Type & Growth vector & Representation & Dynkin diagram \\ [0.5ex] 
 \hline\hline
$(B_n,\p_n)$ & $\bigl(n,\tbinom{n}2\bigr)$ & $(\Gamma_{\pi_1},\Gamma_{\pi_2})$ & 
 \hspace{0.2cm}{\raisebox{-0.02in}{\begin{tikzpicture}
\draw (0.05,0) -- (0.6,0); \draw (1,0) -- (1.6,0); \draw (1.6,0.03) -- (2.5,0.03); \draw (1.6,-0.03) -- (2.5,-0.03); 
\draw (2.0,0.15) -- (2.2,0) -- (2.0,-0.15);
\node[draw,circle,inner sep=1.5pt,fill=white] at (0,0) {}; \node[draw,circle,inner sep=1.5pt,fill=white] at (1.6,0) {}; 
\node at (0.8,0) {$...$}; \node at (2.6,0) {$\times$}; \node at (2,0.35) {}; 
 \end{tikzpicture}}} \\[1.0ex] 
$(G_2,\p_1)$  & $(2,1,2)$ & $(\Gamma_1,\Gamma_0,\Gamma_1)$ & 
 {\raisebox{-0.02in}{\begin{tikzpicture}
\draw (0.05,0) -- (1,0); \draw (0.1,0.05) -- (1,0.05); \draw (0.1,-0.05) -- (1,-0.05); \draw (0.6,0.15) -- (0.4,0) -- (0.6,-0.15);
\node at (0,0) {$\times$}; \node[draw,circle,inner sep=1.5pt,fill=white] at (1,0) {};
 \end{tikzpicture}}} \\[1.0ex]
$(G_2,\p_2)$ & $(4,1)$ & $(\Gamma_3,\Gamma_0)$ & 
 \hspace{0.1cm}{\raisebox{-0.02in}{\begin{tikzpicture}
\draw (0,0) -- (0.95,0); \draw (0,0.05) -- (0.9,0.05); \draw (0,-0.05) -- (0.9,-0.05); \draw (0.6,0.15) -- (0.4,0) -- (0.6,-0.15);
\node[draw,circle,inner sep=1.5pt,fill=white] at (0,0) {}; \node at (1,0) {$\times$}; 
 \end{tikzpicture}}} \\[0.5ex]
$(E_6,\p_2)$ & $(20,1)$ & $(\Gamma_{\pi_3},\Gamma_0)$ & 
 {\raisebox{-0.02in}{\begin{tikzpicture}
\draw (0.0,0) -- (0.5,0); \draw (0.5,0) -- (1.0,0); \draw (1.0,0) -- (1.5,0); \draw (1.5,0) -- (2.0,0); \draw (1.0,0) -- (1.0,0.3); 
\node[draw,circle,inner sep=1.5pt,fill=white] at (0,0) {};\node[draw,circle,inner sep=1.5pt,fill=white] at (0.5,0) {};
\node[draw,circle,inner sep=1.5pt,fill=white] at (1.0,0) {};\node[draw,circle,inner sep=1.5pt,fill=white] at (1.5,0) {};
\node[draw,circle,inner sep=1.5pt,fill=white] at (2.0,0) {}; \node at (1.0,0.35) {$\times$}; 
 \end{tikzpicture}}} \\[0.5ex]
 \hspace{0.1cm}$(E_7,\p_2)$ & $(35,7)$ & $(\Gamma_{\pi_3},\Gamma_{\pi_6})$ & 
 {\raisebox{-0.02in}{\begin{tikzpicture}
\draw (0.0,0) -- (0.5,0); \draw (0.5,0) -- (1.0,0); \draw (1.0,0) -- (1.5,0); \draw (1.5,0) -- (2.0,0); \draw (2.0,0) -- (2.5,0); 
\draw (1.0,0) -- (1.0,0.3); 
\node[draw,circle,inner sep=1.5pt,fill=white] at (0,0) {};\node[draw,circle,inner sep=1.5pt,fill=white] at (0.5,0) {};
\node[draw,circle,inner sep=1.5pt,fill=white] at (1.0,0) {};\node[draw,circle,inner sep=1.5pt,fill=white] at (1.5,0) {};
\node[draw,circle,inner sep=1.5pt,fill=white] at (2.0,0) {};\node[draw,circle,inner sep=1.5pt,fill=white] at (2.5,0) {};
\node at (1.0,0.35) {$\times$}; 
 \end{tikzpicture}}} \\[0.5ex]
 \hspace{0.2cm}$(E_8,\p_2)$ & $(56,28,8)$ & $(\Gamma_{\pi_3},\Gamma_{\pi_6},\Gamma_{\pi_1})$ & 
 {\raisebox{-0.02in}{\begin{tikzpicture}
\draw (0.0,0) -- (0.5,0); \draw (0.5,0) -- (1.0,0); \draw (1.0,0) -- (1.5,0); \draw (1.5,0) -- (2.0,0); \draw (2.0,0) -- (2.5,0); 
\draw (2.5,0) -- (3.0,0); \draw (1.0,0) -- (1.0,0.3); 
\node[draw,circle,inner sep=1.5pt,fill=white] at (0,0) {};\node[draw,circle,inner sep=1.5pt,fill=white] at (0.5,0) {};
\node[draw,circle,inner sep=1.5pt,fill=white] at (1.0,0) {};\node[draw,circle,inner sep=1.5pt,fill=white] at (1.5,0) {};
\node[draw,circle,inner sep=1.5pt,fill=white] at (2.0,0) {};\node[draw,circle,inner sep=1.5pt,fill=white] at (2.5,0) {};
\node[draw,circle,inner sep=1.5pt,fill=white] at (3.0,0) {}; \node at (1.0,0.35) {$\times$}; 
 \end{tikzpicture}}} \\[0.7ex]
 \hline
 \end{tabular}
\end{table}

Since we must have $\g_{-1}=\Gamma_{\pi_1}$ as an $A_{n-1}$ representation, we conclude that only $\m=\g_-$ corresponding to cases 
$(G_2,\p_1)$ and $(B_n,\p_n)$ apply. These cases (as well as the other cases from the table) 
are not prolongation rigid: $\op{pr}(\m)=G_2$ in the first case,
and $\op{pr}(\m)=B_n$ in the second case for $n>2$, while for $n=2$ we get that $\m=\mathfrak{heis}(3)$ is of
infinite type: $\op{pr}(\m)=\mathfrak{cont}(3)$. The claim follows.
 \end{proof}

 \begin{rk}
The parabolics $\p_1$ and $\p_r$ for $A_r$, $\p_r$ for $C_r$, $\p_{r-1}$ and $\p_r$ for $D_r$ correspond 
to $|1|$-grading $\g=\g_{-1}\oplus\g_0\oplus\g_1$, so $\m=\g_{-1}$ is Abelian and the respective $\Delta=TM_o$
is holonomic. Of course, we still have $\g_0=\gl(n)$ and $\op{pr}(\m)=\mathfrak{vect}(n)$ is the algebra of polynomial vector fields in $n$ variables. 

The contact gradation of $\g$, when $\g_0$ has the form $\gl(k)$, is given by parabolic $\p_2$ of $G_2$ and $\p_2$ of $E_6$
($k=2$ or $6$ respectively) and in this case the prolongation of $\m$ is $\mathfrak{cont}(k+1)$.

The only simple Lie algebra that does not allow $\g_0$ of type $\gl(k)$ for any choice of parabolic is $F_4$.
 \end{rk}
  
 \begin{proof}[Proof of Corollary \ref{Cor1}]
The absence of Cauchy characteristics follows from the condition $\g_0=\gl(n)$. Indeed, the space of Cauchy characteristics
$\mathfrak{z}(\m)\cap\g_{-1}$ is $\g_0$ invariant, and due to $\g_0$-irreducibility of $\g_{-1}$ it is either the entire $\m$ 
or 0. The first case is impossible, as $\m$ becomes Abelian while the depth is $s>1$.
Now the symmetry bound is given by the prolongation rigidity and the symmetry bound of \cite{K}.
 \end{proof}
 
Now we generalize Theorem \ref{Th1} as follows. First of all recall that Tanaka's prolongation can be defined for a pair 
$(\m,\g_0)$, where $\g_0$ is a subalgebra of the full algebra of grading-preserving derivations \cite{T}. 
This prolongation is denoted $\op{pr}(\m,\g_0)$; in the case $\g_0=\mathfrak{der}_0(\m)$, it coincides with the previous prolongation.

 \begin{definition}\label{def2}
A fundamental GNLA $\m$ with reduction $\g_0\subset\mathfrak{der}_0(\m)$ is called 
$\g_0$ sub-free if $\g_0\subset\gl(\g_{-1})$ is an irreducible Lie subalgebra. 
(All $\g_{-k}$ are $\g_0$ representations, possibly reducible for $k>1$.)
Maximal such $\m$ % of infinite depth 
is called $\g_0$ free GNLA.
(Any $\g_0$ sub-free Lie algebra is a quotient of $\g_0$ free GNLA by a graded ideal.)
 \end{definition}

Note that irreducibility implies that $\g_0$ is reductive (its representation on $\g_{-1}$ can be real, complex or quaternionic
and irreducibility is assumed as such). Since $\g_0\subset\mathfrak{der}_0(\m)$ 
the subalgebra $\mathfrak{der}_0(\m)\subset\gl(\g_{-1})$ is also reductive and acts on $\g_{-k}$.
However extending $\g_0$ to $\hat{\g}_0=\mathfrak{der}_0(\m)$ and treating $\m$ as $\hat{\g}_0$ sub-free algebra
may increase the prolongation: $\op{pr}(\m,\g_0)\subset\op{pr}(\m,\hat{\g}_0)$.
In particular, $\op{pr}_+(\m,\hat{\g}_0)=0$ $\Rightarrow$ $\op{pr}_+(\m,\g_0)=0$.

A reductive Lie algebra $\g_0$ is a direct sum of a semisimple part $\g_0^{ss}$ and the center. 
In the $\g_0$ sub-free case, 
when $\g_0\subset\gl(\g_{-1})$ is irreducible, $\mathfrak{z}(\g_0)$ has dimension at most two (in the real case; 
over $\C$ it is one).
 A specific real case, when $\g_0^{ss}$ is trivial and $\mathfrak{z}(\g_0)$ is two-dimensional, is considered in \cite{K2}.

A pair $(\m,\g_0)$ is said to be of finite type (or rigid) if $\dim\op{pr}(\m,\g_0)<\infty$. A Tanaka criterion for
finite type \cite[Corollary 2 of Theorem 11.1]{T} is as follows. Define the following ideal in $\g_0$:
 \[
\h_0=\{a\in\g_0\subset\mathfrak{der}_0(\m):a(x)=0\ \forall x\in\g_{-k},k>1\}.
 \]
Then $(\m,\g_0)$ is of finite type iff there are no rank one elements in $\h_0^\C\subset\gl(n,\C)$, $n=\dim\g_{-1}$.

 \begin{lemma}\label{L2}
A $\g_0$ sub-free GNLA $\m$ is of finite type unless $\m$ is Heisenberg.
 % $\g_0=\mathfrak{sp}(2n)$ $\vee$ $\mathfrak{sp}(2n)$ and $\m=\mathfrak{Heis}(2n+1)$.
 \end{lemma}

 \begin{proof}
If $\g_0^{ss}$ is semi-simple, with decomposition $\g_0^1\oplus\dots\oplus\g_0^s$ into simple components, then 
the irreducible representation $\g_{-1}$ is $\Gamma_{\omega_1}\boxtimes\dots\boxtimes\Gamma_{\omega_s}$ 
(or the real part of such a complex product) with factors $\Gamma_{\omega_i}$ being nontrivial $\g_0^i$ irreps.
The semi-simple part of the above ideal $\h_0$ is obtained from $\g_0$ by removal of some simple components $\g_0^i$.
Any element of the remaining components $\g_0^j$ (even in complexification) acts on $\g_{-1}$ by
an endomorphism of rank $>1$ because of additional trivial factors. 
The same concerns the center of $\h_0$ (it does not contain the grading element, but may contain, for instance,
a complex structure). Thus this case leads to the conclusion that $\m$ is of finite type.

Assume now $\g^{ss}_0$ is simple. 
If $\m$ is not Heisenberg, then there is a non-trivial irrep in $\g_{-i}$, $i>1$, and hence $\h^{ss}_0=0$. 
The center of $\g_0$ is one-dimensional (generated by the grading element) if $\g_0\subset\mathfrak{gl}(\g_{-1})$ is complex, 
but it may be two-dimensional if this is a real or quaternionic representation. The additional central element
is a complex structure, so it has full rank. Thus the claim follows from the Tanaka criterion. 
 \end{proof}
  
 \begin{theorem}\label{Th2}
If $\m$ is $\g_0$ sub-free then $\op{pr}_+(\m,\g_0)=0$ unless $\m=\g_-$ 
for a gradation of a simple Lie algebra $\g$ corresponding to a maximal parabolic subalgebra $\p$.
 \end{theorem}

We recall that maximal parabolics in $\g$ are given by a single cross on the Dynkin diagram over $\C$. 
Over $\R$ they are given by a cross on a white node or on a pair of arrow-connected nodes on the Satake diagram.

 \begin{proof}
Let $\g=\op{pr}(\m,\g_0)$. 
Since a Heisenberg algebra $\m$ arises from the parabolic $\p_2$ gradation of $BD$ series and 
from the parabolic $\p_1$ gradation of $C$ series, which are excluded, 
we can conclude from Lemma  \ref{L2} that $\g$ is finite dimensional. 
Let $\g^{ss}$ denote the Levi part of $\g$. Again, by Proposition 2.5 of \cite{Yt},
the prolongation rigidity is equivalent to $\g^{ss}=\g_0^{ss}$. 
That argument about $\g^{ss}$, in the case of complex representation $\g_0\subset\gl(\g_{-1})$,
extends verbatim from the proof of Theorem \ref{Th1}, 
implying the alternative: either $\g_1\neq0$ in which case $\g$ is semi-simple 
(but, in fact, simple due to irreducibility of $\g_{-1}$) or $\g_1=0$ $\Leftrightarrow$ $\g_+=0$. 

In the real case when $\g_{-1}$ becomes $\g_0$-reducible after complexification, the simple modules in it are conjugate.
However $\g^{ss}$ must contain all components $\g_{-1}'\subset\g_{-1}$ that act non-trivially on $\g_1$ and 
the conjugate components behave similarly in this respect; thus we still get $\g_{-1}\subset\g^{ss}$ and the same conclusion: if $\g_+\neq0$ then $\g=\g^{ss}$, if $\g_+=0$ then $\g^{ss}=\g^{ss}_0$. 
We have to consider the first case.

If $\g^{ss}$ is not simple, then $\g_{-1}$ is $\g_0$-reducible, which contradicts the condition that $\g_0$ is sub-free.

If $\g^{ss}$ is simple, then $\g_{-1}$ is $\g_0$-irreducible iff $\p$ is maximal. This corresponds to one cross on the Dynkin
diagram for all complex cases and most real cases $(\g,\p)$, except for $(\mathfrak{su}(p,q),\p_{1,p+q-1})$,
$\bigl(\mathfrak{so}(r-1,r+1),\p_{r-1,r}\bigr)$, $(\mathfrak{so}^*(2r),\p_{r-1,r})$,
$(EII,\p_{1,6})$ and $(EIII,\p_{1,6})$, where $EII,EIII$ are real forms of $E_6$.

The real and complex exceptions are never prolongation rigid: if $\m=\g_-$ for a simple Lie algebra $\g$, 
then $\op{pr}(\m,\g_0)\supset\g$.
The equality is given by the Yamaguchi's theorem \cite{Ym} (which works over $\C$ and over $\R$): 
for the depth $s>1$ the prolongation is bigger (in fact, infinite) only for $(C_r,\p_1)$, in which case 
$\m=\mathfrak{heis}(2r-1)$, $\g_0=\mathfrak{csp}(2r-2)$ and $\op{pr}(\m,\g_0)=\mathfrak{cont}(2r-1)$ is infinite. 
 % The equality is given by the Yamaguchi's theorem \cite{Ym} (works over $\C$ and over $\R$): $\op{pr}(\m)=\g$
 % except for the cases of $|1|$-gradation of $\g$, the contact gradation of $\g$ and in two special cases: 
 % $(A_r,\p_{1,i})$ for $r>2$, $i\neq1,r$, and $(C_r,\p_{1,r})$ for $r>1$.
 % In the latter special cases $\g_0$ is not irreducible on $\g_{-1}$. 
The claim is proven.
 \end{proof}

Let us note that omitting the assumption that $\g_{-1}$ is $\g_0$ irreducible in Definition \ref{def2}
can lead to non-trivial prolongations. For instance, pseudo-product structures \cite{Yt} arising in the geometric theory 
of differential equations are such: they are not prolongation rigid yet the prolongation is not semisimple in general.

 \begin{rk}
For irreducible representations $V$ of a (reductive) Lie algebra $\h$ the problem of prolongation of the pair $(\h,V)$ is classical.
The list of such pairs with $\h^{(1)}\neq0$ was given by E.\ Cartan \cite{C} and Kobayashi-Nagano \cite{KN}. The entries correspond to the above description and $\op{pr}(\m,\g_0)=\g$ unless $(\g,\p)$ is subordinated to $(A_r,\p_1)$,
$\g_0=\mathfrak{gl}(r)\vee\mathfrak{sl}(r)$ or to $(C_r,\p_1)$ with reduction of $\m$ to $\g_{-1}$,
$\g_0=\mathfrak{sp}(2r-2)\vee\mathfrak{csp}(2r-2)$. 
 \end{rk}

Similar to the remark at the end of Section \ref{S2}, all $\g_0$ sub-free GNLAs are obtained by maximal extension
(to the negative side) and constraining to a $\g_0$-submodule in some $\g_{-k}$, then again extension, etc.

 \begin{example}\label{E8a}
For the gradation of $E_8$ arising in the proof of Theorem \ref{Th1}, corresponding to the parabolic $\p_2$, 
we have $\g_0=\gl(8)$. As $\g_0^{ss}=A_7$ module, $\g_{-1}=\Gamma_{\pi_3}$. Then we have:
$\Lambda^2\Gamma_{\pi_3}=\Gamma_{\pi_6}+\Gamma_{\pi_2+\pi_4}$. Thus the minimal extensions are given by
$\g'_{-2}=\Gamma_{\pi_6}$ and $\g''_{-2}=\Gamma_{\pi_2+\pi_4}$. 
 
Next, $\Gamma_{\pi_3}\otimes\Gamma_{\pi_6}=(\Gamma_{\pi_3+\pi_6}+\Gamma_{\pi_2+\pi_7})+\Gamma_{\pi_1}$
and $\Gamma_{\pi_3}\otimes\Gamma_{\pi_2+\pi_4}=(\Gamma_{\pi_3+\pi_6}+\Gamma_{\pi_2+\pi_7}
+\Gamma_{2\pi_2+\pi_5}+\Gamma_{\pi_1+2\pi_4})+
\Gamma_{\pi_4+\pi_5}+\Gamma_{\pi_2+\pi_3+\pi_4}+\Gamma_{\pi_1+\pi_3+\pi_5}+\Gamma_{\pi_1+\pi_2+\pi_6}$.
The parenthetical expressions are in $\Lambda^3\Gamma_{\pi_3}$ and hence have to be eliminated by the Jacobi identity,
whence the maximal extensions are
$\g'_{-3}=\Gamma_{\pi_1}$ and $\g''_{-3}=\Gamma_{\pi_4+\pi_5}
+\Gamma_{\pi_2+\pi_3+\pi_4}+\Gamma_{\pi_1+\pi_3+\pi_5}+\Gamma_{\pi_1+\pi_2+\pi_6}$. 
In particular, choosing $\m'_3=\g'_{-3}\oplus\g'_{-2}\oplus\g'_{-1}$ we get $\op{pr}(\m'_3)=E_8$.
Any other choice of $\m$ from the above components, including $\m'_2=\g'_{-2}\oplus\g'_{-1}$, together with 
the given choice of $\g_0$, leads to prolongation rigid algebra, e.g.\ $\op{pr}_+(\m'_2,\g_0)=0$.
 \end{example} 

 \begin{example}\label{E8b}
For the gradation of $E_8$ corresponding to the parabolic $\p_1$, with growth vector $(64,14)$, 
we have $\g_0=\mathfrak{co}(14)$. Geometrically,
it corresponds to a homogeneous subconformal structure on a 14-dimensional nonholonomic distribution in a manifold $M^{78}$.

As $\g_0^{ss}=D_7$ module, $\g_{-1}=\Gamma_{\pi_7}$. % or $\Gamma_{\pi_6}$ 
Then we have: $\Lambda^2\Gamma_{\pi_7}=\Gamma_{\pi_1}+\Gamma_{\pi_5}$ and this is $\g_{-2}$ for $\g_0$ free GNLA. 
The minimal extensions are given by $\g'_{-2}=\Gamma_{\pi_1}$ and $\g''_{-2}=\Gamma_{\pi_5}$. 
 
Next, $\Gamma_{\pi_7}\otimes\Gamma_{\pi_1}=\Gamma_{\pi_6}+\Gamma_{\pi_1+\pi_7}$
and $\Gamma_{\pi_7}\otimes\Gamma_{\pi_5}=(\Gamma_{\pi_6}+\Gamma_{\pi_4+\pi_6}+\Gamma_{\pi_1+\pi_7})+
\Gamma_{\pi_3+\pi_7}+\Gamma_{\pi_5+\pi_7}+\Gamma_{\pi_2+\pi_6}$.
The parenthetical expression is $\Lambda^3\Gamma_{\pi_7}$ and hence has to be eliminated by the Jacobi identity,
whence the maximal extensions are: trivial for $\g'_{-2}$ and 
$\g''_{-3}=\Gamma_{\pi_3+\pi_7}+\Gamma_{\pi_5+\pi_7}+\Gamma_{\pi_2+\pi_6}$. 
In particular, choosing $\m=\g'_{-2}\oplus\g'_{-1}$ we get $\op{pr}(\m)=E_8$ but any other choice
leads to prolongation rigid algebras.
 \end{example} 
 
 \begin{example}\label{E8c}
For the contact gradation of $E_8$ corresponding to the parabolic $\p_8$, with growth vector $(56,1)$, 
we have $\g_0=\C\oplus E_7$.
As $\g_0^{ss}=E_7$ module, $\g_{-1}=\Gamma_{\pi_7}$. Then we have:
$\Lambda^2\Gamma_{\pi_7}=\Gamma_{0}+\Gamma_{\pi_6}$. Thus minimal extensions are given by
$\g'_{-2}=\Gamma_{0}$ and $\g''_{-2}=\Gamma_{\pi_6}$. 
 
Next, $\Gamma_{\pi_7}\otimes\Gamma_{0}=\Gamma_{\pi_7}$
and $\Gamma_{\pi_7}\otimes\Gamma_{\pi_6}=(\Gamma_{\pi_5}+\Gamma_{\pi_7})+
\Gamma_{\pi_2}+\Gamma_{\pi_1+\pi_7}+\Gamma_{\pi_6+\pi_7}$.
The parenthetical expression is $\Lambda^3\Gamma_{\pi_7}$ and hence has to be eliminated by the Jacobi identity,
whence the maximal extensions are: $\g'_{-3}=\Gamma_{\pi_7}$ and 
$\g''_{-3}=\Gamma_{\pi_2}+\Gamma_{\pi_1+\pi_7}+\Gamma_{\pi_6+\pi_7}$. 
In particular, choosing $\m=\g'_{-2}\oplus\g'_{-1}$ we get $\op{pr}(\m,\g_0)=E_8$ (while $\dim\op{pr}(\m)=\infty$)
but any other choice leads to prolongation rigid algebras.
 \end{example}
 
 \begin{rk}
Note that even though $\m_2'$ in Example \ref{E8a} is smaller than $\m$ (more precisely: $\m_2'$ is the
quotient of $\m$ by $\g_{-3}'$) the prolongation of $(\m,\g_0)$ is larger than that of $(\m_2',\g_0)$. 
Another phenomenon of this kind is attained for the gradation of $E_7$ arising in the proof of Theorem \ref{Th1}:
$\g_0=\gl(7)$, $\m=\g_{-2}\oplus\g_{-1}$, $\g_{-1}=\Gamma_{\pi_3}$, $\g_{-2}=\Gamma_{\pi_6}$.
We have: $\op{pr}(\m)=\op{pr}(\m,\g_0)=E_7$, $\op{pr}(\g_{-1},\g_0)=\g_{-1}\rtimes\g_0$.
 \end{rk}

Finally let us give a geometric implication of the prolongation-rigidity Theorem \ref{Th2}.
Recall that a nonholonomic structure on a vector distribution $\Delta\subset TM$ with symbol $\m$ is a 
$G_0$ structure reduction of the corresponding frame bundle, where $G_0$ is a Lie group with Lie algebra 
$\g_0\subset\mathfrak{der}_0(\m)$, cf.\ \cite{T,Ym,Yt,HM}. In particular, sub-Riemannian, sub-conformal,
pseudo-product and CR are examples of such structures.

 \begin{cor}\label{Cor2}
For the class of nonholonomic distributions $\Delta$ on a manifold $M$ with a geometric structure that is a $G_0$-reduction, 
whose symbols are $\g_0=\op{Lie}(G_0)$ sub-free, there exists a Cartan connection on the frame bundle over $M$, 
which gives an equivariant solution to the equivalence problem. 
 \end{cor}

The existence of a Cartan connection implies that the frame bundle $\mathcal{F}$ becomes a principal $G_0$-bundle,
and the differential invariants of $G_0$-structures descend from $\mathcal{F}$ to the base manifold $M$, 
and hence are free of auxiliary parameters unavoidable with the general construction 
of absolute parallelism by the method of moving frames.
A Cartan connection gives a finer structure of curvatures, allowing to control symmetry reductions in twistor correspondences.
The construction of Cartan connection follows the standard scheme, see e.g.\ \cite{AD,HM}. A $G_0$-equivariant
normalization exists because the structure group $G_0$ is reductive. 
Moreover, the reductive group allows to split the principal bundle $\mathcal{F}\to M$ and hence obtain a refinement:

 \begin{cor}\label{Cor3}
Under the conditions of Corollary \ref{Cor2}, $M$ possesses a canonical linear torsion-free connection. 
 \end{cor}

This in turn yields a possibility for holonomy reductions of the geometry on the base.

\medskip

{\bf Acknowledgment.}
This work was supported by the UiT Aurora project MASCOT, the TFS funded Lie-Størmer Center, 
 %the Norwegian Financial Mechanism 2014-2021 (GRIEG project SCREAM, registration number 2019/34/H/ST1/00636),
and RCN project MatNorVeg: Algebra-Geometri-Analyse med anvendelser \#360114.

%%%%%

\end{document}